\documentclass[11pt]{article}
\usepackage[utf8]{inputenc}
\usepackage{amsthm}
\usepackage{amsmath}
\usepackage{amsfonts} 
\usepackage{amssymb}
\usepackage{dsfont}
\usepackage{enumerate}
\providecommand{\U}[1]{\protect\rule{.1in}{.1in}}

\newtheorem{theorem}{Theorem}[section]
\newtheorem{definition}[theorem]{Definition}
\newtheorem{remark}[theorem]{Remark}
\newtheorem{lemma}[theorem]{Lemma}
\newtheorem{proposition}[theorem]{Proposition}
\newtheorem{corollary}[theorem]{Corollary}

\title{Measure-theoretic sequence entropy pairs and mean sensitivity}
\date{}
\author{Felipe García-Ramos and Víctor Muñoz-López}

\newcommand{\Z}{\mathds{Z}}
\newcommand{\N}{\mathds{N}}

\newcommand{\s}{\mathcal{S}}
\newcommand{\B}{\mathcal{B}}
\newcommand{\diam}{\mbox{diam}}

\let\epsilon\varepsilon

\begin{document}

\maketitle

\abstract{We characterize measure-theoretic sequence entropy pairs of continuous abelian group actions using mean sensitivity. This solves an open question mentioned by Li and Yu \cite{li2021mean}. As a consequence of our results we provide a simpler characterization of Kerr and Li's independence sequence entropy pairs ($\mu$-IN-pairs) when the measure is ergodic and the group is abelian. }

\section{Introduction}
Given a dynamical system (or a group action) with positive entropy one might wonder which points of the phase space contribute to the entropy (and which not). One approach to answer this question was given in Blanchard's seminal paper \cite{blanchard1993disjointness}, where entropy pairs were introduced. In this approach, the answer of this question is not a subset of the phase space $X$ but a subset of the product space $X^2$. It turns out that a system has positive topological entropy if and only if there is an entropy pair. This sets the basis of what is now known as Local entropy theory (for a survey see \cite{glasner2009local}). This theory has provided ground for building new bridges from dynamics to other areas of mathematics like combinatorics \cite{kerr2007independence}, point-set topology \cite{darji2017chaos}, group theory, operator algebras \cite{chung2015homoclinic,barbieri2022markovian} and descriptive set theory \cite{darji2021local}. Of particular interest for this paper is Kerr and Li's characterization of entropy pairs of amenable group actions using the combinatorial notion of independence \cite{kerr2007independence}. 

Positive (measure) entropy can also be localized using measure entropy pairs for a measure (as defined in \cite{blanchard1995entropy}). Furthermore, these  pairs can also be characterized using the concept of independence \cite{Kerr2009}. However, the definition of independence used for this characterization is quite more involved and technical than the topological one (see Section 2.4). The aim of this paper is to bring more clarity to the measure-theoretic theory of independence pairs in the particular case of sequence entropy. 

Sequence entropy was introduced by Kushnirenko and it provided the first link between the functional analytic ergodic theory of von Neumann and the entropy-related ergodic theory of Kolmogorov's school, by proving that a system has pure point spectrum if and only if it has zero sequence entropy \cite{kushnirenko1967metric}. 

As in the classic case, sequence entropy pairs appear exactly when a system has positive sequence entropy \cite{Huang2004}. By replacing arbitrarily large sets instead of positive density, Kerr and Li also characterized sequence entropy pairs using independence \cite{kerr2007independence,Kerr2009}. 

We characterize measure sequence entropy pairs of abelian group actions (Theorem \ref{main}) using the so called mean sensitivity pairs (introduced in \cite{Garcia-Ramos2017}). This solves an open question mentioned in the introduction of \cite{li2021mean}. Measure mean sensitivity is a statistical version of sensitivity to initial conditions used in chaotic dynamics. As an application of our result, we provide a simpler version of independence pairs for sequence entropy (Theorem \ref{independence}). 

There are previous results that indicate that sequence entropy pairs satisfy some type of sensitivity with respect to a measure \cite{Huang2011,li2021density,li2021mean}. Nonetheless, the previous studied notions were too weak to induce a characterization. Theorem \ref{main} generalizes \cite[Theorem 38]{Garcia-Ramos2017}, but since that result does not work with independence some new tools are needed, such as Lemma \ref{lmm:BigNeighbourhoodsIntersection}.

The paper is organized as follows. In Section 2 we give definitions of the main concepts used in the paper (such as entropy, independence and sensitivity). In Section 3 we prove the main technical lemma (Lemma \ref{lmm:BigNeighbourhoodsIntersection}) and the main results of the paper (Theorem \ref{main} and Theotem \ref{independence}). 
Finally, in Section 4 we introduce a stronger definition, diam-mean sensitivity pairs; these pairs are always sequence entropy pairs but we do not know if the converse holds.

\section{Preliminaries}
Throughout this paper $X$ represents a compact metric space with metric $d$. We denote by $\B(X)$ the set of all Borel sets of $X$.
Let $\mu$ a Borel probability measure for $X$. We denote by $\B_{\mu}^{+}(X)$ the set of all Borel sets with positive measure. 

\subsection{Amenable groups}

In this paper $G$ denotes a countable group with identity $e$. 
A sequence $\{F_n\}_{n\in\N}$ of nonempty finite subsets of $G$
is called a \textbf{Følner sequence} if $|sF_n\Delta F_n|/|F_n|\to 0$
as $n\to\infty$ for every $s\in G$. In general, we will simply denote this sequence with $\{F_n\}$. 
A group is called \textbf{amenable} if it admits a Følner sequence. Every abelian group is amenable.

Let $\{F_n\}$ be a Følner sequence and $S\subseteq G$.
We define the lower density of $S$ as
\[\underline{D}(S)=\liminf_{n\to\infty} \frac{|S\cap F_n|}{|F_n|},\]
and the upper density of $S$ as
\[\overline{D}(S)=\limsup_{n\to\infty} \frac{|S\cap F_n|}{|F_n|}.\]

\begin{definition}
Let $ \{F_n\} $ be a Følner sequence of $G$.
We say that $\{F_n\}$ is \textbf{tempered} if there exists $C>0$ such that 
\[\left|\bigcup_{k=1}^{n-1} F_k^{-1}F_n\right|\leq C |F_n|, \mbox{ for all } n>1.\]
\end{definition}

Every Følner sequence has a subsequence that is tempered. 
\subsection{Group actions and sequence entropy}

 
By an action of the group $G$ on $X$ we mean a map $\alpha : G \times X \rightarrow X$ such that, $\alpha(s,\alpha(t,x)) = \alpha(st,x)$ and $\alpha(e,x) = x$ for all $x \in X$ and $s,t\in G$. For simplicity we refer to the action as $G\curvearrowright X$, and we denote the image of a pair $(s,x)$ simply as $sx$. 

Given a continuous group action, we say a Borel probability measure is \textbf{ergodic} if it is $G$-invariant and every $G$-invariant set has measure $0$ or $1$. Throughout this paper every measure is a probability measure and we will omit writing this. 

\begin{theorem}
[\cite{lindenstrauss2001pointwise}]
Let $G\curvearrowright X$ be a continuous action, $\mu$ an ergodic measure, $f$ an integrable function and $\{F_n\}$ a tempered Følner sequence. Then
\[\lim_{n\to\infty}\frac{1}{|F_n|} \sum_{s\in F} f(sx)=\int f(x) d\mu(x) \quad\mu\mbox{-a.e.}\]
\end{theorem}

Let $A$ be a measurable set. We say $x$ is a \textbf{generic point for $A$} is it satisfies the pointwise ergodic theorem for $f=1_A$. 

Let $G\curvearrowright X$ be a continuous action, $\mu$ a $G$-invariant (probability) measure,
$S=\{s_i\}_{i=0}^\infty$ an increasing sequence of non negative integers
and $\mathcal{P}$ a finite measurable partition of $X$.
The \textbf{Shannon entropy} of $\mathcal{P}$ is given by
$$H_\mu(\mathcal{P})=-\sum_{A\in\mathcal{P}}\mu(A)\log\mu(A).$$
The \textbf{sequence entropy of} $\mathcal{P}$ with
respect to $G\curvearrowright X$ and $\mu$ along $S$ is defined by
\[h_\mu^S(X,\mathcal{P})=\limsup \frac{1}{n}
H_\mu\left( \bigvee_{i=0}^{n-1} T^{-s_i}\mathcal{P} \right).\]
The \textbf{sequence entropy} of $G\curvearrowright X$ with respect $\mu$ along $S$ is
\[h_\mu^S(X,G)=\sup_{\mathcal{P}} h_\mu^S(X,\mathcal{P}),\]
where the supremum is taken over all finite measurable partitions $\mathcal{P}$.

Sequence entropy can be studied locally using sequence entropy pairs, introduced in \cite{Huang2004}. 
\begin{definition}
Let $G\curvearrowright X$ be a continuous group action and $\mu$ an invariant measure.
We say that $(x,y)\in X^2$ is a $\mu$-\textbf{sequence entropy pair}
if for any finite measurable partition $\mathcal{P}$ such that there is no $P\in\mathcal{P}$ with $x,y\in\overline{P}$,
there exists $S\subseteq G$ with $h_\mu^S(X,\mathcal{P})>0$.
\end{definition}

\subsection{Mean sensitivity}

Sensitivity with respect to a measure was introduced in \cite{Huang2011}. Mean sensitivity with respect to a measure is a statistical form of sensitivity introduced in \cite{Garcia-Ramos2017}. In contrast with $\mu$-sensitivity, $\mu$-mean sensitivity is invariant under measure isomorphism. 
\begin{definition}
Let $\{F_n\}$ be a Følner sequence of $G$, $G\curvearrowright X$ a continuous group action and $\mu$ a Borel probability measure.
We say that $G\curvearrowright X$ is $\mu$\textbf{-mean sensitive}
if there exists $\epsilon>0$ such that for every $A\in\B_{\mu}^{+}(X)$
there exist $x,y\in A$ such that
\[\limsup_{n\to\infty} \frac{1}{|F_n|}\sum_{s\in F_n} d(sx,sy)>\epsilon.\]
\end{definition}

When $\mu$ is an invariant measure, the previous notion does not depend on the choice of Følner sequence \cite{Garcia-Ramos2019}. 
Now we will define the local notion. 

\begin{definition}
Let $\{F_n\}$ be a Følner sequence of $G$, $G\curvearrowright X$ a continuous group action and $\mu$ an invariant measure.
We say that $(x,y) \in X^2$ is a $\mu$-\textbf{mean sensitivity pair} if
$x\neq y$ and for all open neighbourhoods $U_x$ of $x$ and $U_y$ of $y$,
there exists $\epsilon>0$ such that for every $A\in \mathcal{B}_X^+$
there exist $p,q\in A$ such that
$$\overline{D}(\{s: sp\in U_x \mbox{ and } sq\in U_y\})>\epsilon.$$
We denote the set of $\mu$-mean sensitivity pairs by $S_\mu^m(X,G)$.
\end{definition}



We will now define sensitivity with respect to an $L^2$ function.

\begin{definition}
Let $\{F_n\}$ be a Følner sequence of $G$,  $G\curvearrowright X$ a continuous group action, $\mu$ a Borel probability measure and $f\in L^2(X,\mu)$.
We define
\[d_f(x,y)=\limsup\left(
\frac{1}{|F_n|}\sum_{s\in F_n} |f(sx)-f(sy)|^2\right)^{1/2}.\]
\end{definition}

\begin{definition}
Let $\{F_n\}$ be a Følner sequence of $G$,  $G\curvearrowright X$ a continuous group action, $\mu$ an invariant measure and $f\in L^2(X,\mu)$.
We say that $G\curvearrowright X$ is $\mu$-$f$-\textbf{mean sensitive} if there exists
$\epsilon>0$ such that for every $A\in \B_{\mu}^{+}(X)$ there exist $x,y\in A$
such that $d_f(x,y)>\epsilon$.
In this case we say that $f$ is $\mu$\textbf{-mean sensitive}.
We denote the set of $\mu$-mean sensitive functions with $H_{ms}$
\end{definition}

\begin{remark}
We have that $G\curvearrowright X$ is $\mu$-$1_B$-mean sensitive
if and only if there exists 
$\epsilon>0$ such that for every $A\in\B_{\mu}^{+}(X)$
there exist $p,q\in A$ such that
$\overline{D}(\{s\in G: sp\in B \mbox{ and }sq \in B^c\})>\epsilon$.

\end{remark}

\subsection{Independence}
\begin{definition}
Let $X$ be a set and $(A_1,A_2)$ a pair of subsets of $X$.
Let $E:G\to 2^X$, where $2^X$ is the power set of $X$.
We say that a set $I\subseteq G$ is an \textbf{independence set
for $A$ relative to $E$} if for every nonempty finite subset $F\subseteq I$
and any map $\sigma:F\to\{1,2\}$ we have
$\bigcap_{s\in F}(E(s)\cap s^{-1}A_{\sigma(s)}) \neq \emptyset$.
\end{definition}

We denote by $\B'_{\mu}(X,\epsilon)$ the set of all maps $E:G\to\B(X)$ such that
$\mu(E(s))\geq 1-\epsilon$, for all $s\in G$.
\begin{definition}
\label{def:indfiniteset}
Let $G\curvearrowright X$ be a continuous group action and $\mu$ an invariant measure.
For $A_1,A_2\subseteq X$ and $\epsilon>0$
we say that $(A_1,A_2)$ has $(\epsilon,\mu)$-\textbf{independence
over arbitrarily large finite sets}
if there exists $c>0$ such that for every $N>0$
there is a finite set $F\subseteq G$ with $|F|>N$
such that for every $E\in\B'_{\mu}(X,\epsilon)$
there is an independence set $I\subseteq F$ for $(A_1,A_2)$ relative to $E$
with $|I|\geq c|F|$.
\end{definition}
\begin{definition}
\label{def:INpair}
Let $G\curvearrowright X$ be a continuous group action and $\mu$ an invariant measure. We say that $(x,y)\in X^2$ is an \textbf{$IN_{\mu}$ pair}
if for every neighbourhoods $U_x$ of $x$ and $U_y$ of $y$
there exists $\epsilon>0$ such that
$(U_x,U_y)$ has $(\epsilon,\mu)$-independence
over arbitrarily large finite sets.
\end{definition}

\begin{theorem}[\cite{Kerr2009}]
\label{thm:INpair}
Let $G$ be an amenable group, $G\curvearrowright X$ a continuous group action and $\mu$ an invariant measure. A pair $(x,y)\in X^2$
is an $IN_{\mu}$ pair if and only if it is a $\mu$-sequence entropy pair.
\end{theorem}

\subsection{Almost periodicity}
Let $G\curvearrowright X$ be a continuous group action and $\mu$ an invariant measure.
We define the \textbf{Koopman representation}  $\kappa:G \to \mathcal{B}(L^2(X,\mu))$
by $\kappa(s)f(x)=f(s^{-1}x)$, for all $s\in G$, $f\in L^2$ and $x\in X$,
where $\mathcal{B}(L^2(X,\mu))$ is the space of all bounded linear operators on $L^2(X,\mu)$.

\begin{definition}
Let $G\curvearrowright X$ be a continuous group action, $\mu$ an invariant measure and $f\in L^2(X,\mu)$.
We say that $f$ is \textbf{almost periodic} function
if $\overline{\kappa(G)(f)}$ is compact as a subset of $L^2(X.\mu)$.
We denote by $H_{ap}$ the set of almost periodic functions.
\end{definition}


\begin{theorem}\cite[Theorem 1.15]{Garcia-Ramos2019}
\label{thm:AperiodicFunc=MeanSensFunc}
Let $G$ be an abelian group, $\{F_n\}$ a Følner sequence, $G\curvearrowright X$ a continuous group action and $\mu$ an ergodic measure.
We have that $H_{ms}\subset H_{ap}^c$. Furthermore, if  $\{F_n\}$ is a tempered Følner sequence, then $H_{ms}= H_{ap}^c$. 
\end{theorem}

\section{Characterization of sequence entropy pairs}
\subsection{Mean sensitivity pairs}

\begin{lemma}[\cite{kushnirenko1967metric}]
\label{lmm:Kushnirenko}
Let $(X,\mu)$ be a probability space
and $\{\xi_n\}$ be a sequence of two-set partitions of $X$,
with $\xi_n=\{P_n,P_n^c\}$.
The closure of $\{1_{P_1},1_{P_2},\ldots\}\subseteq L^2(X,\mu)$ is compact
if and only if for all subsequences $m_i$
\[\lim_{n\to \infty} \frac{1}{n}H\left(\bigvee_{i=1}^n \xi_{m_i}\right)=0.\]
\end{lemma}

The next theorem is obtained directly from Lemma \ref{lmm:Kushnirenko}.

\begin{theorem}\label{thm:AperiodicFunc=SeqEntr}
Let $G\curvearrowright X$ be a continuous group action, $\mu$ an ergodic measure and $B\in \B_{\mu}^{+}(X)$.
Then $1_B\in H_{ap}$ if and only if $h_\mu^{\s}(X,\{A,A^c\})=0$
for any sequence $S\subseteq G$.
\end{theorem}

\begin{proposition}
\label{prop:ida}
Let $\{F_n\}$ be a Følner sequence of $G$, $G\curvearrowright X$ a continuous group action and $\mu$ an ergodic measure. If $(x,y)\in X^2$ is a
$\mu$-mean sensitivity pair then it is a $\mu$-sequence entropy pair.
\end{proposition}

\begin{proof}
Let $(x,y)$ be a $\mu$-mean sensitivity pair and $\mathcal{P}=\{P,P^c\}$
a finite partition, such that $x\in P\setminus P^c$ and $y\in P^c\setminus P$.
This implies that there exist neighbourhoods $U_x$ of $x$ and $U_y$ of $y$,
such that $U_x\subseteq P$ and $U_y\subseteq P^c$.
Note that $G\curvearrowright X$ is $\mu$-$1_P$-diam-mean sensitive.
Hence, by Theorem \ref{thm:AperiodicFunc=MeanSensFunc}
and Theorem \ref{thm:AperiodicFunc=SeqEntr} we obtain that
there exists $S\subseteq \N$ such that $h_\mu^S(\mathcal{P},T)>0$.
Thus, $(x,y)$ is a $\mu$-sequence entropy pair.
\end{proof}
The following lemma is standard, e.g., see \cite[Proposition 5.8]{Huang2011}. 

\begin{lemma}\label{pro:BigSetsIntersection}
Let $(X,\B(X),\mu)$ be a Borel probability space, $a>0$,
and $\{E_s\}_{s\in G}$ a sequence of measurable sets
with $\mu(E_s)\geq a$, for all $s\in G$.
There exists $N$ such that for any set $F\subseteq G$ with $|F|\geq N$
there exist $s_F,t_F\in F$ such that
$\mu(E_{s_F}\cap E_{t_F})>0$.
\end{lemma}

The following lemma is the main tool used to provide a connection between independence over finite sets and sensitivity. 

\begin{lemma}\label{lmm:BigNeighbourhoodsIntersection}
Let $G\curvearrowright X$ be a continuous group action, $\mu$ an ergodic measure, $U_x,U_y\subseteq X$ open sets and $\varepsilon>0$. 
If $(U_x,U_y)$ has
$(\epsilon,\mu)$-independence over arbitrarily large finite sets, then for every $A\in \B_{\mu}^{+}(X)$, there exists $(s_A,t_A)\in R_A$
such that $\mu(s_A^{-1}U_x\cap t_A^{-1}U_y)\geq \epsilon$, where
\[R_A=\{(s,t)\in G^2:\mu(s^{-1}A\cap t^{-1}A)>0\}.\]
\end{lemma}

\begin{proof}
We will prove this by contradiction. Assume that there exists $A\in \B_{\mu}^{+}(X)$ such that $$\mu(s^{-1}U_x\cap t^{-1}U_y)<\epsilon$$
for all $(s,t)\in R_A$.
This implies that $\mu((s^{-1}U_x\cap t^{-1}U_y)^c)\geq 1-\epsilon$
for all $(s,t)\in R_A$.

By the independence hypothesis, there exists $c>0$ which satisfies the
definition of $(\epsilon,\mu)$-independence
over arbitrarily large finite sets  (Definition \ref{def:indfiniteset}).
Using Lemma \ref{pro:BigSetsIntersection} on $\{E_s\}_{s\in G}$ with $E_s=s^{-1}A$
we conclude that there is $N_0>0$ such that for any finite set $F\subseteq G$
with $|F|>N_0$ there exist $s_F,t_F\in F$ such that
\[\mu(s_F^{-1}A\cap t_F^{-1}A)>0.\]
Using Definition \ref{def:indfiniteset}, there exists a finite set $F_0$ with
$|F_0|\geq N_0/c$ such that for every $E\in\B'_{\mu}(X,\epsilon)$
there is an independence set $I\subseteq F_0$ such that
$|I|\geq c |F_0|\geq N_0$.
This implies that for every $\sigma:I\to \{x,y\}$ we have that

\begin{equation}
\label{ind}
\bigcap_{g\in I}E(g)\cap g^{-1}U_{\sigma(g)}\neq \emptyset.
\end{equation}

Furthermore, since $|I|\geq N_0$, there exists $s_I,t_I\in I$ such that
\[\mu(s_I^{-1}A\cap t_I^{-1}A)>0.\]

Let $E:G\to 2^X$ be the constant function with \[E(g)=(s_I^{-1}U_x\cap t_I^{-1}U_y)^c.\]
(The fact that it's constant will be important for Theorem \ref{independence}.)
Note that $E\in \B'_{\mu}(X,\epsilon)$.
Let $\sigma:I\to \{x,y\}$
with $\sigma(s_I)=x$ and $\sigma(t_I)=y$.
Then
\begin{align*}
\bigcap_{g\in I}E(g)\cap g^{-1}U_{\sigma(g)}
&\subseteq (E(s_I)\cap s_I^{-1}U_x) \cap (E(t_I)\cap t_I^{-1}U_y)\\
&= (s_I^{-1}U_x\cap t_I^{-1}U_y)^c \cap s_I^{-1}U_x\cap t_I^{-1}U_y\\
&=\emptyset.
\end{align*}
This is a contradiction to (\ref{ind}).
Therefore, there exists  $(s_A,t_A)\in R_A$ such that $\mu(s_A^{-1}U_x\cap t_A^{-1}U_y)>\epsilon$.
\end{proof}

A couple of relationships (but no characterization) between sequence entropy pairs and some type of sensitivity pairs are known.
In \cite[Theorem 5.9]{Huang2011} it was shown that every sequence entropy pair is a $\mu$-sensitivity pair. In \cite[Theorem 1.8]{li2021density} it was shown that sequence entropy pairs are density-sensitivity pairs. 

Now we provide a characterization.
\begin{theorem}
\label{main}
Let $G$ be an abelian group, $G\curvearrowright X$ a continuous group action and $\mu$ an ergodic measure.
The following are equivalent
\begin{enumerate}
    \item $(x,y)\in X^2$ is a $\mu$-sequence entropy pair.
    \item $(x,y)\in X^2$ is $\mu$-mean sensitivity pair with respect to a tempered  Følner sequence.
    \item $(x,y)\in X^2$ is a $\mu$-mean sensitivity pair with respect to every tempered Følner sequence.
\end{enumerate}

\end{theorem}

\begin{proof}

Given 2.) we obtain 1.) with Proposition \ref{prop:ida}.
It only remains to prove that 1.) implies 3.). 
Let $(x,y)\in X^2$ be a $\mu$-sequence entropy pair,
$U_x$ and $U_y$ neighbourhoods of $x$ and $y$, and $A\in\B_{\mu}^{+}(X)$.
From Theorem \ref{thm:INpair}, there exists $\epsilon>0$ such that $(U_x,U_y)$ has
$(\epsilon,\mu)$-independence over arbitrarily large finite sets.
Using Lemma \ref{lmm:BigNeighbourhoodsIntersection},
there exists $(s,t)\in R_A$ such that $$\mu(s^{-1}U_x\cap t^{-1}U_y)>\epsilon.$$

Let $z\in X$ be a generic point of $s^{-1}U_x\cap t^{-1}U_y$ and $s^{-1}A\cap t^{-1}A$g,
$e\geq 0$ the first entry time of $z$ in $s^{-1}A\cap t^{-1}A$, $p=sez$ and $q=tez$.
Note that $p,q\in A$. 
If $gez\in s^{-1}U_x\cap t^{-1}U_y$ then
$gp=gsez\in U_x$ and $gq=gtez\in U_y$.
Therefore, for any tempered Følner sequence we have that
\[\overline{D}(\{g:gp \in U_x\mbox{ and } gq \in U_y\})
=\mu(s^{-1}U_x\cap t^{-1}U_y)\geq \epsilon.\]
\end{proof}

\subsection{Independence}
As a consequence of our techniques we provide a simpler characterization of $IN_{\mu}$ pairs when $\mu$ is an ergodic measure. 


We define $\B_{\mu}(X,\epsilon)=\{D\in \B(X):\mu(D)\geq 1-\varepsilon\}$.

The following definition is very similar to Kerr and Li's definition of measure IN-pairs but we use $\B$ instead of $\B'$. 
\begin{definition}
Let $G\curvearrowright X$ be a continuous group action and $\mu$ an invariant measure.

For $A_1,A_2\subseteq X$ and $\epsilon>0$
we say that $(A_1,A_2)$ has $\B$-$(\epsilon,\mu)$-\textbf{independence
over arbitrarily large finite sets}
if there exists $c>0$ such that for every $N>0$
there is a finite set $F\subseteq\Z_+$ with $|F|>N$
such that for every $D\in\B_{\mu}(X,\epsilon)$
there is an independence set $I\subseteq F$ for $(U_x,U_y)$ relative to $D$
with $|I|\geq c|F|$.

We say that $(x,y)$ is a \textbf{$\B$-$IN_{\mu}$ pair}
if for every neighbourhoods $U_x$ of $x$ and $U_y$ of $y$
there exists $\epsilon>0$ such that
$(U_x,U_y)$ has $\B$-$(\epsilon,\mu)$-independence 
over arbitrarily large finite sets.
\end{definition}

\begin{theorem}
\label{independence}
Let $G$ be an abelian group, $G\curvearrowright X$ a continuous group action and $\mu$ an ergodic measure. Then $(x,y)$ is an $IN_{\mu}$ pair if and only if it is a $\B-IN_{\mu}$ pair. 
\end{theorem}
\begin{proof}
One direction is trivial. The other can be obtained by noting that in the proof of Lemma \ref{lmm:BigNeighbourhoodsIntersection} we actually only use $\B-IN_{\mu}$ pairs. 
\end{proof}

\section{Diam-mean sensitivity pairs}

Diam-mean sensitivity was introduced in \cite{Garcia-Ramos2017} and can be used to characterize when a maximal equicontinuous factor is 1-1 for $\mu$-almost every point \cite{garcia2021mean}. In this section we introduce the measure theoretic version of this concept. We provide some basic results which are adaptations from results in \cite{Garcia-Ramos2017,Huang2011}. Nonetheless, this new notion is still somewhat mysterious to the authors since we don't know if its invariant under isomorphism or not.

\begin{definition}
Let $G\curvearrowright X$ be a continuous group action and $\mu$ an invariant measure.
We say that $G\curvearrowright X$ is $\mu$\textbf{-diam-mean sensitive}
if there exists $\epsilon>0$ such that for every $A\in\B_{\mu}^{+}(X)$
we have that
\[\limsup_{n\to\infty} \frac{1}{|F_n|}\sum_{s\in F_n} diam(sA)>\epsilon.\]
\end{definition}

Note that if $G\curvearrowright X$ is $\mu$-mean sensitive then it is $\mu$-diam-mean
sensitive. 

We do not know if there exists a $\mu$-diam mean sensitive system with discrete spectrum. 

\begin{definition}

Let $G\curvearrowright X$ be a continuous group action and $\mu$ an invariant measure.
We say that $(x,y) \in X^2$ is a $\mu$-\textbf{diam-mean sensitivity pair} if
$x\neq y$ and for all open neighbourhoods $U_x$ of $x$ and $U_y$ of $y$,
there exists $\epsilon>0$ such that for every $A\in \mathcal{B}_X^+$
there exists $S\subseteq \Z_+$ with $\overline{D}(S)>\epsilon$
such that for every $s\in S$ there exist $p,q\in A$
such that $sp \in U_x$ and $sq \in U_y$.
We denote the set of $\mu$-diam-mean sensitivity pairs by $S_\mu^{dm}(X,G)$.
\end{definition}

\begin{proposition}
Let $G\curvearrowright X$ a continuous group action and $\mu$ an invariant measure. We have that $G\curvearrowright X$ is $\mu$-diam-mean sensitive
if and ony if $S_\mu^{dm}(X,G)\neq \emptyset$.

\end{proposition}

\begin{proof}

Suppose that $G\curvearrowright X$ is $\mu$-diam-mean sensitive
with sensitive constant $\epsilon_0>0$ and $\epsilon\in(0,\epsilon_0)$.
We consider the follow compact set
\[X^\epsilon=\{(x,y)\in X^2:d(x,y)\geq\epsilon\}.\]
Suppose that $S_\mu^{dm}(X,G)=\emptyset$.
This implies that for every $(x,y)\in X^\epsilon$
there exist open neighbourhoods of $x$ and $y$, $U_x$ and $U_x$,
such that for every $\delta>0$ there exists $A_\delta(x,y)\in \B_X^+$ such that
\[\overline{D}(\{s\in G:\exists p,q\in A_\delta(x,y) \mbox{ such that }
sp\in U_x, sq\in U_y\})\leq \delta.\]
There exists a finite set $F\subset X^\epsilon$ such that 
\[X^\epsilon\subseteq \bigcup_{(x,y)\in F} U_x\times U_y.\]
Let $\delta=\epsilon/|F|$ and
$S(x,y)=\{s\in G:\exists p,q\in A_\delta(x,y)\mbox{ such that }(sp,sq)\in U_x\times U_y\}$, for every $(x,y)\in F$.
Since $\mu$ is ergodic, for every $(x,y)\in F$ there exists $t(x,y)\in G$
such that $A:=\bigcap_{(x,y)\in F} t(x,y)A_\delta(x,y)\in \B^+_X$. Note that
\begin{align*}
S:&=\{s\in G:\exists p,q\in A\mbox{ such that } (sp,sq)\in X^\epsilon\}\\
&\subseteq\bigcup_{(x.y)\in F} \{s\in G:\exists p,q\in A\mbox{ such that } (sp,sq)\in U_x\times U_y\}\\
&\subseteq \bigcup_{(x,y)\in F} S(x,y)t(x,y)^{-1}.
\end{align*}
Hence $\overline{D}(S)\leq |F|\delta=\epsilon$.

On the other hand, since $G\curvearrowright X$ is $\mu$-diam-mean sensitive
and $\epsilon$ is smaller than the sensitive constant we have that
$\overline{D}(S)> \epsilon.$ That is a contradiction.

Now we prove the other direction.

Suppose that $S_\mu^{dm}(X,G)\neq \emptyset$. Let $(x,y)\in S_\mu^{dm}(X,G)$.
There exist neighbourhoods $U_x$ of $x$ and $U_x$ of $y$, with $d(U_x,U_y)>0$,
and $\epsilon>0$, such that for every $A\in \B_X^+$
there exists $S\subseteq G$ with $\overline{D}(S)>\epsilon$
such that for every $s\in S$ there exist $p,q\in A$
such that $sp\in U_x$ and $sq\in U_y$.
This implies that
$\overline{D}(\{s\in G:\diam(sA)>d(U_x,U_y)\})\geq \overline{D}(S)>\epsilon$.
Therefore, $G\curvearrowright X$ is $\mu$-diam-mean sensitive.
\end{proof}

A consequence of Theorem \ref{main} is the following. 
\begin{corollary}
Let $G$ be an abelian group, $G\curvearrowright X$ a continuous group action and $\mu$ an ergodic measure. Every $\mu$-sequence entropy pair is a $\mu$-diam-mean sensitivity pair.
\end{corollary}

We do not know if the converse of the previous corollary holds. 




\bibliographystyle{plain}
\bibliography{ref.bib}

\noindent Felipe Garc\'ia-Ramos,  fgramos@conacyt.mx. 
{\em  CONACyT, Universidad Aut\'onoma de San Luis Potos\'i, and Jagiellonian University.}\\

\noindent V\'ictor Mu\~noz-L\'opez, vmmunozlopez@gmail.com,
{\em Instituto de F\'isica, Universidad Aut\'onoma de San Luis Potos\'i, Mexico}
\end{document}